\documentclass[12pt]{amsart}
\usepackage{amsmath}
\usepackage{commath}

\usepackage[english]{babel}
\usepackage{amssymb}

\usepackage[dvips]{graphicx}
 \usepackage{enumerate}
\usepackage{mathrsfs}
\usepackage{units}
\usepackage{ stmaryrd }

\usepackage[colorlinks=true, pdfstartview=FitV, linkcolor=blue,
citecolor=blue, urlcolor=blue]{hyperref}

\usepackage[utf8x]{inputenc}

\usepackage{tgtermes}
\usepackage[T1]{fontenc}

\calclayout

\theoremstyle{plain}
\newtheorem{theorem}{Theorem}[section]

\newtheorem{example}[theorem]{Example}

\theoremstyle{definition}
 \theoremstyle{definition}

\newtheorem{remark}[theorem]{Remark}

\numberwithin{equation}{section}

\allowdisplaybreaks

\newcommand{\R}{\mathbb R}

\DeclareMathOperator{\supp}{supp}

\DeclareMathOperator*{\essinf}{ess\,inf}
\DeclareMathOperator*{\esssup}{ess\,sup}

\newcommand{\pp}{{p(\cdot)}}
\newcommand{\cpp}{{p'(\cdot)}}
\newcommand{\Lp}{L^{p(\cdot)}}
\newcommand{\Pp}{\mathcal P}
\newcommand{\qq}{{q(\cdot)}}

\DeclareMathAlphabet\EuRoman{U}{eur}{m}{n}
\SetMathAlphabet\EuRoman{bold}{U}{eur}{b}{n}
\renewcommand{\mathsf}{\EuRoman}

\author{D.~Cruz-Uribe, OFS}
\address{Department of Mathematics \\ University of Alabama \\
Tuscaloosa, AL 35487, USA}
\email{dcruzuribe@ua.edu}

\author{A.\ Fiorenza}
\address{Dipartimento di Costruzioni e Metodi Matematici in
Architettura \\
Universit\`a di Napoli \\ Via Monteoliveto, 3 \\
I-80134 Napoli, Italy\\
and Istituto per le Applicazioni del Calcolo
"Mauro Picone", sezione di Napoli \\
Consiglio Nazionale delle Ricerche \\
via Pietro Castellino, 111 \\
 I-80131 Napoli, Italy }
\email{fiorenza@unina.it}

\author{O.~M.~Guzm\'an}
\address{Departamento de Matemáticas\\ Universidad Nacional de Colombia \\ AP360354 Bogotá \\ Colombia}
\email{omguzmanf@unal.edu.co}

\title{Embeddings between grand, small and variable Lebesgue spaces}

\thanks{The first author is supported by NSF Grant 1362425 and
  research funds from the Dean of the College of Arts \& Sciences, the
  University of Alabama.  The authors would like to thank the
  anonymous referee for his/her very helpful comments.}

\date{June 18, 2017}

\begin{document}

%----------------------------------------------------------------------------------------------------------------------------------%
\begin{abstract}
We give conditions on the exponent function $\pp$ that imply the
existence of
embeddings between the grand, small and variable Lebesgue spaces.  We
construct examples to show that our results are close to optimal.  Our
work extends recent results by the second author, Rakotoson and
Sbordone~\cite{MR3485874}. 
\end{abstract}
%------------------------------------------------------------------------------------------------------------------------------------%

\maketitle

%------------------------------------------------------------------------------------------------------------------------------%

\section{Introduction}
\label{sec:intro}

In this paper we consider the relationship between three Banach
function spaces that generalize the classical Lebesgue spaces.  Given
a set $\Omega\subset \R^n$, $|\Omega|=1$, $1<p<\infty$, and
$\theta>0$, the generalized grand Lebesgue space
$L^{p),\theta}(\Omega)$ consists of all measurable functions $f$ such
that
\[ \|f\|_{p),\theta} = \sup_{0<\epsilon<p-1}
  \bigg(\epsilon^\theta\int_\Omega
  |f(x)|^{p-\epsilon}\,dx\bigg)^{\frac{1}{p-\epsilon}}.  \]
When $\theta=0$ this reduces to the Lebesgue space $L^p(\Omega)$.
When $\theta=1$, this becomes the grand Lebesgue space
$L^{p)}(\Omega)$, which was introduced by Iwaniec and
Sbordone~\cite{MR1176362}.  The generalization for $\theta>0$ was
introduced in~\cite{Greco:1997uf} and considered in a more systematic
way in~\cite{MR2110048}.  These spaces have proved very useful in
proving limiting results in the study of partial differential
equations: see~\cite{MR2836342,MR3122331,MR3311355,
  Greco:1997uf,MR3278982}.  

The small Lebesgue space $L^{(p,\theta}$  is defined as the associate space
of $L^{p'),\theta}$, and so has the norm
\[ \|f\|_{(p,\theta} = \sup\bigg\{ \int_\Omega f(x)g(x)\,dx :
  \|f\|_{p'),\theta} \leq 1\bigg\}. \]
An intrinsic expression for the small Lebesgue space norm, when
$\theta=1$, was first
found in~\cite{MR1776829} and for general $\theta>0$
in~\cite{MR2110048}.  These expressions were quite complicated, but
much simpler expressions were found in~\cite{MR2499725,MR2110397}:
\begin{gather}
\label{eqn:grand-norm}
\|f\|_{p),\theta} \approx
\sup_{0<t<1} \log\bigg(\frac{e}{t}\bigg)^{-\frac{\theta}{p}}
\left(\int_t^1 f_*(s)^p\,ds\right)^{\frac{1}{p}} \\
\label{eqn:small-norm}
\|f\|_{(p,\theta} \approx
\int_0^1 \log\left(\frac{e}{t}\right)^{\frac{\theta}{p'}-1}
\left(\int_0^t f_*(s)^p\,ds\right)^{\frac{1}{p}}\frac{dt}{t}.
\end{gather}

 The grand and small Lebesgue spaces are very ``close'' to the space
 $L^p$.  More precisely,  we have for all
$1<p<\infty$ and $\epsilon>0$ that 
\begin{equation} \label{eqn:grand-small-embed}
L^{p+\epsilon}(\Omega) \subsetneq  L^{(p,\theta}(\Omega) \subsetneq L^p(\Omega)
\subsetneq L^{p),\theta}(\Omega) \subsetneq L^{p-\epsilon}(\Omega).
\end{equation}
The parameter $\theta$ controls the ``distance'' of these spaces from
$L^p$: for instance, if $\Omega=[0,1]$, then we have that
\begin{equation} \label{eqn:in-grand}
 \left(\frac{1}{t}\right)^{\frac{1}{p}}\log
  \left(\frac{e}{t}\right)^{\frac{\theta-1}{p}} \in L^{p),\theta},
\end{equation}
and for all non-negative $\epsilon,\,\delta$, $\epsilon+\delta>0$,
\[ \left(\frac{1}{t}\right)^{\frac{1}{p}}\log
  \left(\frac{e}{t}\right)^{-(1+\epsilon)\frac{1}{p}
-(1+\delta)\frac{\theta}{p'}}
  \in L^{(p,\theta}. \]
The first inclusion follows from the proof
of~\cite[Proposition~5.6]{MR2110048}.  The second uses the
proposition itself, which states that for $\beta>1$,
$L^p(\log L)^{\beta\theta(p-1)}\subset L^{(p,\theta}$.  For this
embedding, see also~\cite{Cobos:2012ue}.  

\medskip

The variable Lebesgue spaces generalize the classical Lebesgue spaces
in a different way.  Given a measurable function $p(\cdot) : \Omega
\rightarrow [1,\infty)$, we define $L^\pp(\Omega)$ to be the
collection of all measurable functions such that for some $\lambda>0$, 
\[ \rho(f/\lambda) = \int_\Omega
  \bigg(\frac{|f(x)|}{\lambda}\bigg)^{p(x)}\,dx < \infty.  \]
$\Lp(\Omega)$ becomes a Banach function space with the norm
\[ \|f\|_\pp = \inf\{ \lambda>0 : \rho(f/\lambda)\leq 1 \}.  \]
These spaces were introduced by Orlicz in the 1930s and have been
extensively studied for the past 25 years.
(See~\cite{cruz-fiorenza-book} for more information on their history
and applications.)   If we define
\[ p_- = \essinf_{x\in \Omega} p(x), \quad 
 p_+ = \esssup_{x\in \Omega} p(x), \]
then
\begin{equation} \label{eqn:var-embed}
 L^{p_+}(\Omega) \subset L^\pp(\Omega) \subset L^{p_-}(\Omega). 
\end{equation}
Equality holds if and only if $p_-=p_+$: i.e., when $\pp$ is constant.

\medskip

Given the embeddings~\eqref{eqn:grand-small-embed}
and~\eqref{eqn:var-embed}, it is a natural question to ask if 
the stronger embeddings $L^\pp(\Omega) \subset L^{(p_-,\theta}(\Omega)$ and
$L^{p_+),\theta}(\Omega) \subset L^\pp(\Omega)$ are possible.  Our first
result gives a sufficient condition for these inclusions to hold.

To state it, we  introduce some notation.  Let $\Pp(\Omega)$
denote the set of all measurable exponent functions $\pp :
\Omega\rightarrow [1,\infty)$.   Given $\pp\in \Pp(\Omega)$, let $p_*(\cdot) :
[0,1] \rightarrow [1,\infty)$ denote the decreasing
rearrangement of $\pp$.  More precisely, define the
distribution function
\[ \mu_{p(\cdot)}(t) = |\{ x\in \Omega : p(x)>t \}|, \]
and define the decreasing rearrangement by
\[ p_*(t) = \inf\{ \lambda \ge 0 : \mu_{p(\cdot)}(\lambda) \leq t \}, \]
where the infimum of the empty set is defined to be $+\infty$.  Let
$p^*(\cdot)$ denote the increasing rearrangement, defined by
\[ p^*(t) = - (-p(\cdot))_*(t) = p_*(1-t). \]
Note that if we modify $\pp$ on a set of measure zero, we may assume
without loss of generality that $p_-=p^*(0)=p_*(1)$ and
$p_+=p_*(0)=p^*(1)$.  Moreover, we have that $p_*(t)\rightarrow
p_*(0)$ as $t\rightarrow 0^+$.  

Given $\pp$, $1<p_-\leq p_+<\infty$, define the conjugate exponent
function $\qq=\cpp$ by $\frac{1}{p(t)}+\frac{1}{q(t)}=1$.  By taking
rearrangements we see that 
\[ \frac{1}{p_*(t)}+\frac{1}{q^*(t)}=1, \]
and the same equality holds for $p^*$ and $q_*$.   We also have
that $p_-'=(p_-)' = p^*(0)'=q_*(0)$.  

\begin{theorem} \label{thm:var-small-embed}
Given an exponent $\pp\in \Pp(\Omega)$, $1<p_-\leq p_+ < \infty$, and
$\theta>0$, 
suppose that there exists $0<t_0\leq 1$ and $\epsilon>0$ such that for all
$t\in [0,t_0]$, 
\begin{equation} \label{eqn:var-small-embed1}
\frac{1}{p^*(0)}-\frac{1}{p^*(t)} 
\geq \left(\frac{\theta}{p_-'}+\epsilon\right)
\frac{\log\log(\frac{e}{t})}{\log(\frac{e}{t})}. 
\end{equation}
Then 
\begin{equation}  \label{eqn:var-small-embed2}
 L^\pp(\Omega) \hookrightarrow L^{(p_-,\theta}(\Omega).
\end{equation}
\end{theorem}

As a consequence of Theorem~\ref{thm:var-small-embed} and the abstract
properties of Banach function spaces we get the second desired
inclusion.

\begin{theorem} \label{thm:grand-var-embed}
Given an exponent $\pp\in \Pp(\Omega)$, $1<p_-\leq p_+ < \infty$, and
$\theta>0$, 
suppose that there exists $0<t_0 \leq 1$ and $\epsilon>0$ such that for all
$t\in [0,t_0]$, 
\begin{equation} \label{eqn:grand-var-embed1}
\frac{1}{p_*(t)}-\frac{1}{p_*(0)} 
\geq \left(\frac{\theta}{p_+}+\epsilon\right)
\frac{\log\log(\frac{e}{t})}{\log(\frac{e}{t})}. 
\end{equation}
Then
\begin{equation} \label{eqn:grand-var-embed2} 
L^{p_+),\theta}(\Omega) \hookrightarrow L^\pp(\Omega). 
\end{equation}
\end{theorem}

We can adapt the proof of Theorem~\ref{thm:var-small-embed} to get
another scale of weaker continuity conditions on the exponent $\pp$ for the desired
embedding to hold.  We will discuss this immediately after the proof:
see Remark~\ref{remark:better-cond} below.  However, the continuity conditions in Theorem~\ref{thm:var-small-embed} and~\ref{thm:grand-var-embed} are 
in some sense sharp, as the next result shows.

\begin{example} \label{example:no-var-small}
Given $\theta>0$, there exists an increasing function
$\pp \in \Pp([0,1])$, $1<p_-\leq p_+<\infty$,  such that for $t\in
[0,e^{-2}]$, 
\begin{equation} \label{eqn:no-var-small}
 \frac{1}{p(0)}-\frac{1}{p(t)} 
\leq \frac{\theta}{p_-'}
\frac{\log\log(\frac{e}{t})}{\log(\frac{e}{t})}, 
\end{equation}
and there exists $f \in L^\pp([0,1])$ such that $f\not\in
L^{(p_-,\theta}([0,1]))$. 
\end{example}

We could also consider the reverse inclusions: for which $\pp$ do we
have that $L^{(p_-,\theta}(\Omega) \subset L^\pp(\Omega)$ or
$L^\pp(\Omega) \subset L^{p_+),\theta}(\Omega)$?  However, while both
inclusions are true if $\pp$ is constant, if $p_-<p_+$, then neither
can hold.  By the same associate space argument as we use to prove
Theorem~\ref{thm:grand-var-embed} below, it suffices to show that the
second inclusion can never hold.  In this case, if $p_-<p_+$, there
exists a set $E\subset \Omega$, $|E|>0$, such that
$p_+(E)=\esssup_{x\in E} p(x)<p_+$.  But then there exists a function
$f$ such that $\supp(f)\subset E$, $f\in L^{p_+(E)}(E)$, and such that
for any $\delta>0$, $f\not\in L^{p_+(E)+\delta}(E)$.  Hence,
$f\in L^\pp(\Omega)$ (see \cite[Corollary~2.50]{cruz-fiorenza-book}),
but not in $L^{p_+),\theta}$: by definition, if
$f\in L^{p_+),\theta}$, then $f\in L^{p_+-\epsilon}$,
$0<\epsilon<p_+-1$.

We can, however, prove a weaker result if we pass to the
``rearranged'' variable exponent spaces considered
in~\cite{MR2373739,MR3485874}.   They showed the following result.

\begin{theorem} \label{thm:FRS}
Given $\pp \in \Pp(\Omega)$, $1<p_-\leq p_+<\infty$,  there exist
constants $c_1,\,c_2>0$ such that for every
$u\in L^\pp(\Omega)$, 
\[ c_1\|u_*\|_{p^*(\cdot)} \leq \|u\|_\pp \leq
  c_2\|u_*\|_{p_*(\cdot)} \leq \infty. \]
\end{theorem}

Note that the last inequality can be an equality:  $u\in L^\pp$ does
not imply $u_*\in
L^{p_*(\cdot)}$:  see~\cite[Remark after proof of Theorem~3]{MR2373739}.   But with
this stronger hypothesis we have the following embedding.

\begin{theorem} \label{thm:rearranged} 
Given $\pp\in \Pp(\Omega)$,
  $1<p_-\leq p_+<\infty$,  and $\theta\geq 1$, suppose there exist  $A \in \mathbb{R}$
  and $0<t_{0}\leqslant 1$ such that for all $t\in [0,t_0]$,

\begin{equation}\label{eqn:rearranged1}
\frac{1}{p_{\ast}(t)}-\frac{1}{p_{\ast}(0)}
\leq
\frac{A}{\log(\frac{e}{t})}
+\frac{\theta-1}{p_{\ast}(0)}\frac{\log\log(\frac{e}{t})}{\log(\frac{e}{t})}.
\end{equation}
Then for all $u\in \Lp(\Omega)$ such that
$u_{\ast}\in L^{p_{\ast}(\cdot)}([0,1])$, $u_{\ast} \in L^{p_+),\theta}([0,1])$.
\end{theorem}

\begin{remark}
When $\theta=1$, Theorem~\ref{thm:rearranged} was proved
in~\cite[Theorem~1]{MR3485874}.  Our proof generalizes and simplifies
theirs.    Note that in this case we must assume $A>0$.  
\end{remark}

\begin{remark}
We conjecture that some version of Theorem~\ref{thm:rearranged} is true for
$0<\theta<1$, but we have  not been able to prove it.  Note that for
$\theta<1$ the condition~\eqref{eqn:rearranged1} is never possible:
the lefthand side is positive, but the righthand side is negative for
all $t$ sufficiently close to~$0$.
\end{remark}

\begin{remark}
  We conjecture that if~\eqref{eqn:rearranged1} holds, then a ``dual''
  result holds as well.  More precisely, we conjecture that given any
  decreasing function $u_*\in L^{(q_-,\theta}$, we have
  $u_*\in L^{q^*(\cdot)}$.  However, unlike in the proof of
  Theorem~\ref{thm:grand-var-embed}, we cannot use associativity to
  prove this since we are not dealing with a subspace but rather the
  cone of decreasing functions.  Moreover, our other techniques do not
  seem applicable to this case.
\end{remark}

The condition~\eqref{eqn:rearranged1} is close to optimal as the
following example shows.

\begin{example} \label{example:no-rearrange}
Given $\pp\in \Pp(\Omega)$,
  $1<p_-\leq p_+<\infty$,  suppose there exists
  $\theta>0$, $\epsilon>0$ and $0<t_{0}\leqslant 1$, such that for all
  $t\in [0,t_0]$,
\begin{equation}\label{eqn:no-rearrange1}
\frac{1}{p_{\ast}(t)}-\frac{1}{p_{\ast}(0)}
\geq
\bigg(\frac{\theta +\epsilon}{p_{\ast}(0)}\bigg) 
\frac{\log\log(\frac{e}{t})}{\log(\frac{e}{t})}.
\end{equation}
Then there exist a (decreasing) function $f_*\in L^{p_*(\cdot)}([0,1])
\setminus L^{p_+),\theta}([0,1])$.  
\end{example}

\begin{remark} If we compare the two
  conditions~\eqref{eqn:rearranged1} and~\eqref{eqn:no-rearrange1}  when $\theta=1$, we see that they differ by a factor of
  $\log\log(\frac{e}{t})$.  It is not clear if this gap can be closed
  or if either condition is optimal.
\end{remark}

We can extend Theorem~\ref{thm:rearranged} to the range $0<\theta<1$,
and generalize it for $\theta\geq 1$, if we pass to a larger scale of
spaces.  Given a decreasing function $\sigma_* : [0,1]\rightarrow \R$,
define the function $\varphi : [0,1]\times \R \rightarrow [0,\infty)$
by
\[ \varphi(a,b) = b^{p_*(a)}\log(e+b)^{\sigma_*(a)}. \]
We define the space $L^{\varphi(\cdot)}([0,1])$ to consist of all measurable
functions $f_*$ defined on $[0,1]$ such that for some $\lambda>0$,
\[ \rho_\varphi(f/\lambda) =\int_0^1
  \varphi\left(t,\frac{|f(t)|}{\lambda}\right)\,dt<\infty. \]
With a norm defined as above for the variable Lebesgue spaces,
$L^{\varphi(\cdot)}([0,1])$ becomes a Banach function space, a
particular case of the Musielak-Orlicz spaces, also referred to as
generalized Orlicz spaces.  With this definition of $\varphi$, these
spaces were first considered in~\cite{MR2471932} and were later
considered by other authors:
see~\cite{cruz-fiorenza-book,diening-harjulehto-hasto-ruzicka2010} for
details and further references.  

\begin{theorem} \label{thm:gen-orlicz}
Given $\theta>0$, let $\sigma_*(\cdot) : [0,1]\rightarrow \R$ be a
bounded, decreasing function such that $\sigma_*(0)\geq 1-\theta$.  Suppose
further that there exists $B>0$ and $0<t_0\leq 1$ such that  for $t\in [0,t_0]$,
\begin{equation} \label{eqn:gen-orlicz1}
\sigma_*(0) -\sigma_*(t) \leq \frac{B}{\log\log(\frac{e}{t})}.  
\end{equation}
Given $\pp\in \Pp(\Omega)$, $1<p_-\leq p_+<\infty$, suppose there
exists $A\in \R$ such that for $t\in [0,t_0]$,
\begin{equation} \label{eqn:gen-orlicz2}
 \frac{1}{p_*(t)}-\frac{1}{p_*(0)} 
\leq \frac{A}{\log(\frac{e}{t})} + \frac{\theta-1+\sigma_*(0)}{p_*(0)}
\frac{\log\log(\frac{e}{t})}{\log(\frac{e}{t})}. 
\end{equation}
Let $\varphi(a,b)=b^{p_*(a)}\log(e+b)^{\sigma_*(a)}$.  Then, for all
$u\in L^\pp(\Omega)$ such that $u_*\in L^{\varphi(\cdot)}([0,1])$,
$u_*\in L^{p_+),\theta}([0,1])$.
\end{theorem}

\begin{remark}
  If $\sigma_*(\cdot)\equiv 0$, then Theorem~\ref{thm:gen-orlicz}
  reduces to Theorem~\ref{thm:rearranged}.
  Theorem~\ref{thm:gen-orlicz} is a more general result:  for example, when
  $\theta>1$, if $\sigma_*(\cdot)\equiv 1-\theta$, then
  $L^{p_*(\cdot)}([0,1]) \subsetneq
  L^{\varphi(\cdot)}([0,1])$. (See~\cite[Chapter II.8]{MR724434}.)
\end{remark}

\medskip

The proofs of our results are given in the next section.  Throughout,
our notation is standard; for variable Lebesgue spaces we follow the
notation established in~\cite{cruz-fiorenza-book}.   Constants
$C,c,\ldots$ may vary in value from line to line.
If we write $A\lesssim B$, we mean that
there exists a constant $c$ such that $A\leq cB$; the constant $c$ can
depend on the exponent function $\pp$ and other fixed parameters, but
it does not depend on any variables in functions or summations.  If
$A\lesssim B$ and $B\lesssim A$, we write $A\approx B$.

\section{Proofs of Results}
\label{sec:proofs}

\begin{proof}[Proof of Theorem~\ref{thm:var-small-embed}]
Define the exponent $r_{\ast}(\cdot)\in \Pp([0,1])$ by
 \begin{equation*}
 \frac{1}{p^*(0)}=\frac{1}{p_{-}}=\frac{1}{p^{\ast}(t)}+\frac{1}{r_{\ast}(t)}.
 \end{equation*}
 Fix $f\in \Lp(\Omega)$.  By the generalized Hölder inequality in the
 scale of the variable Lebesgue
 spaces~\cite[Corollary~2.28]{cruz-fiorenza-book},
 inequality~\eqref{eqn:small-norm}, and the first inequality in
 Theorem~\ref{thm:FRS}, we have

\begin{multline*}
\norm{f}_{(p_{-},\theta}
\lesssim
\int_{0}^{1}{\norm{f_{\ast}\chi_{(0,t)}}_{p_{-}}\frac{dt}{t\log(\frac{e}{t})^{1-\frac{\theta}{p^{\prime}}}}}\\ 
 \lesssim
 \int_{0}^{1}{\norm{f_{\ast}}_{p^{\ast}(\cdot)}\norm{\chi_{(0,t)}}_{r_{\ast}(\cdot)}\frac{dt}{t\log(\frac{e}{t})^{1-\frac{\theta}{p^{\prime}}}}}
 \lesssim
 \norm{f}_{p(\cdot)}\int_{0}^{1}{\norm{\chi_{(0,t)}}_{r_{\ast}(\cdot)}\frac{dt}{t\log(\frac{e}{t})^{1-\frac{\theta}{p^{\prime}}}}}. 
\end{multline*}

Observe that since $t\leq 1$, by~\cite[Corollorary~2.23]{cruz-fiorenza-book}
\begin{equation*}
 \norm{\chi_{(0,t)}}_{r_{\ast}(\cdot)} \leqslant t^{\frac{1}{r_{\ast}(t)}}.
\end{equation*} 
Further, if we exponentiate  our
hypothesis~\eqref{eqn:var-small-embed1}, for $t\in [0,t_0]$ we get
\begin{equation} \label{eqn:hyp-exp}
 \left(\frac{e}{t}\right)^{\frac{1}{p^*(0)}-\frac{1}{p^*(t)}}
\geq \log\left(\frac{e}{t}\right)^{\frac{\theta}{p_-'}+\epsilon}.  
\end{equation}
Therefore, we have that 
\begin{multline*}
\norm{f}_{p(\cdot)}\int_{0}^{1}{\norm{\chi_{(0,t)}}_{r_{\ast}(\cdot)}\frac{dt}{t\log(\frac{e}{t})^{1-\frac{\theta}{p_-^{\prime}}}}} 
\leq 
\norm{f}_{p(\cdot)}\int_{0}^{1}{\frac{t^{\frac{1}{r_{\ast}(t)}}}{t\log(\frac{e}{t})^{1-\frac{\theta}{p_-^{\prime}}}}\,dt} \\
= \norm{f}_{p(\cdot)}\int_{0}^{1}{\frac{t^{\frac{1}{p^*(0)}-\frac{1}{p^{\ast}(t)}}}{t\log(\frac{e}{t})^{1-\frac{\theta}{p_-^{\prime}}}}\,dt}
\lesssim
\norm{f}_{p(\cdot)}\int_{0}^{1}{\frac{dt}{t\log(\frac{e}{t})^{1+\epsilon}}}
\lesssim \norm{f}_{p(\cdot)};
\end{multline*}
for the second to last inequality we use~\eqref{eqn:hyp-exp} for $t\in
[0,t_0]$; for $t\in [t_0,1]$ we use that all of these functions are
bounded and bounded away from $0$.  Combining these two inequalities, we see
that~\eqref{eqn:var-small-embed2} holds.
\end{proof}

\begin{remark} \label{remark:better-cond}
If we replace the hypothesis~\eqref{eqn:var-small-embed1} with the
weaker assumption that for some $\epsilon>0$ and $t\in [0,t_0]$,
\[ \frac{1}{p^*(0)}-\frac{1}{p^*(t)} \geq 
\frac{\theta}{p_-'}\frac{\log\log(\frac{e}{t})}{\log(\frac{e}{t})}
+(1+\epsilon) \frac{\log\log\log(\frac{e}{t})}{\log(\frac{e}{t})}, \]
then the above argument goes through with almost no change, except
that in the final inequality the last integral becomes
\begin{equation} \label{eqn:weaker}
\int _0^1
  \frac{dt}{t\log(\frac{e}{t})\log\log(\frac{e}{t})^{1+\epsilon}}
< \infty. 
\end{equation}
Note, however, that if we take $\epsilon=0$ in \eqref{eqn:weaker}, we
do not recapture the sharp endpoint
condition~\eqref{eqn:no-var-small}.

Even weaker sufficient conditions can be found by finding larger
functions that are still in $L^1([0,1])$ and working backwards through
the proof.  Details are left to the interested reader.
\end{remark}

\medskip

\begin{proof}[Proof of Theorem~\ref{thm:grand-var-embed}]
This is an immediate consequence of
Theorem~\ref{thm:var-small-embed}  and the abstract properties of
Banach function spaces.  Given $\pp$, recall that we let $\qq=\cpp$ be
the dual exponent.   Suppose~\eqref{eqn:grand-var-embed1} holds; then
a straightforward calculation shows that~\eqref{eqn:var-small-embed1}
holds for $\qq$.  Hence, we have that $L^\qq(\Omega) \subset
L^{(q_-,\theta)}(\Omega)$.  

Given a Banach function space $X$, let $X'$
denote its associate space.  Then $L^\qq(\Omega)'= L^\pp(\Omega)$
(\cite[Proposition~2.37]{cruz-fiorenza-book}) and since
$(L^{p_+),\theta})'= L^{(p_+',\theta}=L^{(q_-,\theta}$, we have that 
$(L^{(q_-,\theta)})'=(L^{p_+),\theta})''=L^{p_+),\theta}$
(\cite[Theorem~2.7]{bennett-sharpley88}).  By
\cite[Proposition~2.10]{bennett-sharpley88}, if $Y$ is another Banach
function space such that $X\subset Y$, then $Y'\subset X'$.  Hence, we
get~\eqref{eqn:grand-var-embed2} as desired. 
\end{proof}

%%%%%%%%%%%%%%%%%%%%%%%%%%%%%%%%%%%%%%%%%%

\begin{proof}[Construction of Example~\ref{example:no-var-small}]

For $t\in [0,e^{-2}]$, define
\[ p(t)=2+2\theta\frac{\log\log(\frac{e}{t})}{\log(\frac{e}{t})}. \]
(Note that we could replace $2$ by any value $q>1$; however, for
clarity we will restrict ourselves to this special case.)
Then we have that $p(0)=p_-=2$ and $\pp$ is increasing.  To get an
increasing exponent function on $[0,1]$ we can extend $\pp$ to be
constant on $[e^{-2},1]$.    A straightforward calculation shows that
for $t\in [0,e^{-2}]$, $\pp$ satisfies~\eqref{eqn:no-var-small}:

\[ \frac{1}{p(0)}-\frac{1}{p(t)}
= \frac{\theta}{2}\frac{\log\log(\frac{e}{t})}{\log(\frac{e}{t})+\log\log(\frac{e}{t})}\\
\leqslant \frac{\theta}{p_-'}\frac{\log\log(\frac{e}{t})}{\log(\frac{e}{t})}.
\]

Now fix $1<b<2$ and for $t\in [0,1/e^{2}] $ define
\begin{equation*}
f(t)=\sum_{j=2}^{\infty}{a_{j}\chi_{(e^{-j-1},e^{-j}]}(t)}, \qquad 
a_{j}=\left[ \frac{e^{j}}{j\log(j)^{b}}\right]^{\frac{1}{2}\frac{j+1}{j+1+\theta\log(j+2)}}.
\end{equation*}
Extend $f$ to be constant on $[e^{-2},1]$; then $f$ is a decreasing
function on $[0,1]$.  We will first show that
$f\in L^\pp([0,1])$ and then show that
$f\not\in L^{(p_-,\theta}([0,1])$.  We first estimate $f$ on
$[0,e^{-2}]$:
\begin{align*}
\int_{0}^{e^{-2}}{f(t)^{p(t)}dt}
&=\sum_{j=2}^{\infty}{\int_{e^{-j-1}}^{e^{-j}}{a_{j}^{2+\frac{2\theta\log\log(\frac{e}{t})}{\log(\frac{e}{t})}}dt}}\\
&\leqslant
  \sum_{j=2}^{\infty}\int_{e^{-j-1}}^{e^{-j}} a_{j}^{2+\frac{2\theta \log(j+2)}{j+1}}dt\\
& \lesssim
  \sum_{j=2}^{\infty}e^{-j}
a_{j}^{2+\frac{2\theta  \log(j+2)}{j+1}}\\
& = \sum_{j=2}^{\infty}{\frac{1}{j\log (j)^{b}}}< \infty.
\end{align*}
Since $f$ and $\pp$ are constant on $[e^{-2},1]$, it follows that $f
\in L^{p(\cdot)}([0,1])$. (This follows from the definition of the norm
since $p_+<\infty$.)

\medskip

We now prove that $f\not\in L^{(p_-,\theta}([0,1])$.   Fix
$t \in (0,\frac{1}{e^{2}}]$; there exists $j>2$ such that
$e^{-j}<t \leq e^{-j+1}$.  But then we have that

\begin{align*}
  \norm{f\chi_{(0,t)}}_{2}^{2}
  &=    \int_{0}^{t}\left(\sum_{k=2}^{\infty}
    a_{k}\chi_{(e^{-k-1},e^{-k}]}(s) \right)^{2}\,ds\\ 
  &\geqslant \sum_{k=2}^{\infty}{\int_{0}^{e^{-j}}{a_{k}^{2}\chi_{(e^{-k-1},e^{-k}]}(s)}}\,ds\\
  &= \sum_{k=j}^{\infty}{\int_{e^{-k-1}}^{e^{-k}}{a_{k}^{2}\,ds}}\\
  &\gtrsim \sum_{k=j}^{\infty}{a_{k}^{2}e^{-k}}\\
  & = \sum_{k=j}^{\infty}
\frac{e^{-\frac{\theta k \log(k+2)}{k+1+\theta\log(k+2)}}}
{k^{1-\frac{\theta\log(k+2)}{k+1+\theta\log(k+2)}}
\log(k)^{\frac{b(k+1)}{k+1+\theta\log(k+2)}}} \\
& = \sum_{k=j}^{\infty}
\frac{e^{-\theta\log(k+2)}
e^{\frac{\theta\log(k+2)+\theta^2\log(k+2)^2}{k+1+\theta\log(k+2)}}}
{k^{1-\frac{\theta\log(k+2)}{k+1+\theta\log(k+2)}}
\log(k)^{\frac{b(k+1)}{k+1+\theta\log(k+2)}}} \\
& \gtrsim  \sum_{k=j}^{\infty}
\frac{1}{k^{1+\theta-\frac{\theta\log(k+2)}{k+1+\theta\log(k+2)}}
\log(k)^{b-\frac{b\theta\log(k+2)}{k+1+\theta\log(k+2)}}} \\
& \gtrsim  \sum_{k=j}^{\infty}
\frac{1}{k^{1+\theta}\log(k)^b}.
\end{align*}

To estimate the final sum we will compare it to the corresponding
integral.  For $x\geq 2$,
\[
(\theta+b)\int_x^\infty \frac{dt}{t^{1+\theta}\log(t)^b}
\geq \int_x^\infty 
\left(\frac{\theta}{t}+\frac{b}{t\log(t)}\right)
\frac{1}{t^{1+\theta}\log(t)^b}\,dt
=\frac{1}{x^\theta\log(x)^b}. \]
    
Combining these inequalities we get that
\[
 \norm{f\chi_{(0,t)}}_{2}^{2}
\gtrsim  \sum_{k=j}^{\infty} \frac{1}{k^{1+\theta}\log(k)^b}
\gtrsim \frac{1}{j^\theta \log(j)^b}. \]

Therefore, by~\eqref{eqn:small-norm} we have that
\begin{align*}
\norm{f}_{(p_-,\theta}
& \gtrsim\int_{0}^{e^{-2}}
\norm{f\chi_{(0,t)}}_{2}\frac{1}{t\log(\frac{e}{t})^{1-\frac{\theta}{2}}}\,dt\\
&=
  \sum_{j=3}^{\infty}{\int_{e^{-j}}^{e^{-j+1}}{\norm{f\chi_{(0,t)}}_{2}
\frac{1}{t\log(\frac{e}{t})^{1-\frac{\theta}{2}}}}}\,dt\\
&\gtrsim  \sum_{j=3}^{\infty}\int_{e^{-j}}^{e^{-j+1}}
\frac{1}{j^{\frac{\theta}{2}}\log(j)^{\frac{b}2}}
\frac{1}{e^j\log(e^{j+1})^{1-\frac{\theta}{2}}}\,dt \\
& \gtrsim  \sum_{j=3}^{\infty}
\frac{1}{j\log(j)^{\frac{b}{2}}} 
 = \infty.
\end{align*}
The last sum is infinite because $1<b<2$.  
Thus, we have that  $f \not\in L^{(p_-,\theta}([0,1])$.
\end{proof}

%%%%%%%%%%%%%%%%%%%%%%%%%%%%%%%%%%%%%%%%%%%%%%%%%%%%%%%%%%%%%%%

\begin{proof}[Proof of Theorem~\ref{thm:rearranged}]
If we rearrange~\eqref{eqn:rearranged1} as 
\begin{equation*}
\left(\frac{1}{p_{\ast}(t)}-\frac{1}{p_{\ast}(0)}\right)\log\left(\frac{e}{t}\right)
\leqslant A+\frac{\theta-1}{p_{\ast}(0)}\log\log\left(\frac{e}{t}\right)
\end{equation*}
and exponentiate, we get that for $0<t\leq t_0$, 
\begin{equation*}
\left(\frac{e}{t}\right)^{\frac{1}{p_{\ast}(t)}-\frac{1}{p_{\ast}(0)}}
\leq e^{A}\log\left(\frac{e}{t}\right)^{\frac{\theta-1}{p_{\ast}(0)}}.
\end{equation*}

Since $u_*\in L^{p_*(\cdot)}([0,1])$, by the proof of
\cite[Theorem~1]{MR3485874} we have that for $0\leq t\leq 1$, 
\[u_{\ast}(t) \leqslant
  C\left(\frac{e}{t}\right)^{\frac{1}{p_{\ast}(t)}}.  \]
Therefore, if we combine these two inequalities, we get that for
$0<t\leq t_0$, 
\[ u_{\ast}(t)
  \leqslant C\left(\frac{e}{t}\right)^{\frac{1}{p_{\ast}(0)}}
\log\left(\frac{e}{t}\right)^{\frac{\theta-1}{p_{\ast}(0)}}. \]
Since $u_*(t)$ is bounded and the righthand side is bounded away from
$0$ for $t_0\leq t\leq 1$, the same inequality holds (with a possibly
larger constant $C$) for all $t$.  Therefore, 
by~\eqref{eqn:in-grand} we have that
$u_{\ast}(t)\in L^{p_+),\theta}([0,1])$. 
\end{proof}

%%%%%%%%%%%%%%%%%%%%%%%%%%%%%%%%%%%%%%%%%%%%%%%%%%%%%%%%%%%%%%%%%%%%%%%%%%%%%%%%%%%%%%%%%%%%%%%%%%%%%%%%

\begin{proof}[Construction of Example~\ref{example:no-rearrange}]
Define the function $f_*$ on $[0,t_0]$, $0<t_0<1$, by
\begin{equation*}
f_*(t)=\left[\frac{\log\log(\frac{e}{t})}{t\log(\frac{e}{t})^{1-\theta}}\right]^{\frac{1}{p_{\ast}(0)}}.
\end{equation*}
Note that for all $\theta>0$ there exists $t_0>0$ such that $f_*$ is
decreasing on $[0,t_0]$; extend $f_*$ to be constant on $[t_0,1]$ to
get a decreasing
function on $[0,1]$.  Without loss of generality we may assume that
this is the same $t_0$ as in the hypotheses.  

We will first show that $f_{\ast} \notin L^{p_+),\theta}([0,1])$
and then prove that $f_{\ast}\in L^{p_{\ast}(\cdot)}([0,1])$.  
We may assume that  $t_{0}$ is small enough that there exists $c>0$
so that for $t \in [0,t_0]$, 
\begin{equation}
\frac{c(1+\log\log(\frac{e}{t}))}{t\log(\frac{e}{t})^{1-\theta}}
\leqslant \frac{\log\log(\frac{e}{t})}{t\log(\frac{e}{t})^{1-\theta}},
\end{equation}
But then for $t<t_0$ sufficiently close to $0$,
\begin{multline*}
\int_{t}^{1}{f_*(s)^{p_{\ast}(0)}ds}
\geqslant \int_{t}^{t_{0}}{\frac{\log\log(\frac{e}{s})}{s\log(\frac{e}{s})^{1-\theta}}ds}
\geqslant \int_{t}^{t_{0}}{\frac{c(1+\log\log(\frac{e}{s}))}{s\log(\frac{e}{s})^{1-\theta}}}\\
=
c\log\log\left(\frac{e}{t}\right)\log\left(\frac{e}{t}\right)^{\theta}
- c\log\log\left(\frac{e}{t_0}\right)\log\left(\frac{e}{t_0}\right)^{\theta}
\geq c\log\log\left(\frac{e}{t}\right)\log\left(\frac{e}{t}\right)^{\theta}.
\end{multline*}
Hence, for all $t\in [0,1]$, 
\begin{equation*}
\log\left(\frac{e}{t}\right)^{-\theta}\int_{t}^{1}{f(s)^{p_{\ast}(0)}ds}\geqslant C\log\log\left(\frac{e}{t}\right),
\end{equation*}
and so by~\eqref{eqn:grand-norm},
$f_{\ast}(0)\notin L^{p_{+),\theta}}$.\\

\medskip
 
To prove that $f_{\ast}\in L^{p_{\ast}(\cdot)}([0,1])$, first note that
if we rearrange and exponentiate~\eqref{eqn:no-rearrange1}, we get
\[
\left(\frac{e}{t}\right)^{\frac{p_{\ast}(t)}{p_{\ast}(0)}}
\leq 
\left(\frac{e}{t}\right)\log\left(\frac{e}{t}\right)^{-\frac{p_{\ast}(t)}{p_{\ast}(0)}(\theta+\epsilon)}.
\]
Since $p_*(t)\rightarrow p_*(0)$ as $t\rightarrow 0$, if necessary by
taking $t_0>0$ smaller, we may assume that there exist
$0<\sigma<\tau<\epsilon$ such that for all $t\in
[0,t_0]$,
\[ \frac{p_*(t)}{p_*(0)}(1+\epsilon-\tau) > 1+\sigma. \]

Given these two inequalities we can estimate as follows:
 \begin{multline*}
 \int_{0}^{t_0}f_*(t)^{p_{\ast}(t)}\,dt 
=
\int_{0}^{t_0} \bigg[\frac{\log\log(\frac{e}{t})}{t\log(\frac{e}{t})^{1-\theta}}\bigg]
^{\frac{p_{\ast}(t)}{p_{\ast}(0)}}\,dt
\lesssim
\int_{0}^{t_0}\bigg[\frac{\log(\frac{e}{t})^\tau}{t\log\left(\frac{e}{t}\right)^{1-\theta}}
\bigg]^{\frac{p_{\ast}(t)}{p_{\ast}(0)}}\,dt\\
\lesssim \int_{0}^{t_0}\frac{dt}{t\log\left(\frac{e}{t}\right)
^{\frac{p_{\ast}(t)}{p_{\ast}(0)}(\theta+\epsilon+1-\theta-\tau)}}
\leq \int_{0}^{t_0}\frac{dt}{t\log\left(\frac{e}{t}\right)
^{1+\sigma}}<\infty. 
\end{multline*}
Given this, and since $f_*$ is bounded on $[t_0,1]$,
 we conclude that $f_{\ast}\in L^{p_{\ast}(\cdot)}([0,1])$.
 \end{proof}

%%%%%%%%%%%%%%%%%%%%%%%%%%%%%%%%%%%%%

\begin{proof}[Proof of Theorem~\ref{thm:gen-orlicz}]
  We begin by generalizing an argument in the proof
  of~\cite[Theorem~1]{MR3485874}.    First note that for $b>1$, $\varphi(a,b)$ is
  decreasing in $a$, and that for fixed $a$ it is increasing in
  $b$ for all $b$ sufficiently large.  Second, we may assume without loss of
  generality that $u_*$ is unbounded, since otherwise the desired
  inclusion holds trivially.  But then there exists
  $0<t_0\leq 1$ such that for all $0<s\leq t\leq t_0$,
  $u_*(s)\geq 1$ and  $\varphi(t,u_*(s))$ is decreasing in $s$.   Therefore, since
  $u_*\in L^\varphi([0,1])$, there exists a constant $c>0$ such that for all such $t$,
\begin{multline*} 
c\geq \int_0^t u_*(s)^{p_*(s)}\log(e+u_*(s))^{\sigma_*(s)}\,ds \\
\geq \int_0^t u_*(s)^{p_*(t)}\log(e+u_*(s))^{\sigma_*(t)}\,ds
\geq tu_*(t)^{p_*(t)}\log(e+u_*(t))^{\sigma_*(t)}. 
\end{multline*}

We want to rearrange this inequality to dominate $u_*(t)$.  If we fix
$a$, then the inverse of $\varphi(a,b)$ as a function of $b$ is
\[ \varphi^{-1}(a,b) \approx b^{\frac{1}{p_*(a)}}
\log(e+b)^{-\frac{\sigma_*(a)}{p_*(a)}}. \]
The implicit constants depend on $p_*(a)$ and $\sigma_*(a)$.
Therefore, since these functions are bounded, we have that there
exists an absolute constant $C$ such that 
\[ u_*(t) \leq C\left(\frac{c}{t}\right)^{\frac{1}{p_*(t)}}
\log\left(e+\frac{c}{t}\right)^{-\frac{\sigma_*(t)}{p_*(t)}}
\leq C\left(\frac{e}{t}\right)^{\frac{1}{p_*(t)}}
\log\left(\frac{e}{t}\right)^{-\frac{\sigma_*(t)}{p_*(t)}}. \]

By possibly taking $t_0$ closer to $0$, we have, by rearranging and
exponentiating \eqref{eqn:gen-orlicz1} and \eqref{eqn:gen-orlicz2},
that for all $t\in [0,t_0]$,
\[ \log\left(\frac{e}{t}\right)^{\sigma_*(0)-\sigma_*(t)} \leq e^B, \]
and
\[ \left(\frac{e}{t}\right)^{\frac{1}{p_*(t)}-\frac{1}{p_*(0)}}
\leq e^A \log\left(\frac{e}{t}\right)^{\frac{\theta-1+\sigma_*(0)}{p_*(0)}}. \]
If we combine these three inequalities, we see that for $t\in
[0,t_0]$,
\begin{align*}
u_*(t) 
& \leq C\left(\frac{e}{t}\right)^{\frac{1}{p_*(t)}}
\log\left(\frac{e}{t}\right)^{-\frac{\sigma_*(t)}{p_*(t)}} \\
& \leq  C\left(\frac{e}{t}\right)^{\frac{1}{p_*(0)}}
\log\left(\frac{e}{t}\right)
^{\frac{\theta-1}{p_*(0)}+\frac{\sigma_*(0)}{p_*(0)}-\frac{\sigma_*(t)}{p_*(t)}}
  \\
& \leq  C\left(\frac{e}{t}\right)^{\frac{1}{p_*(0)}}
\log\left(\frac{e}{t}\right)
^{\frac{\theta-1}{p_*(0)}} 
\log\left(\frac{e}{t}\right)
^{\frac{\sigma_*(0)}{p_*(0)}-\frac{\sigma_*(t)}{p_*(0)}}
\log\left(\frac{e}{t}\right)
^{\frac{\sigma_*(t)}{p_*(0)}-\frac{\sigma_*(t)}{p_*(t)}}\\
& \leq  C\left(\frac{e}{t}\right)^{\frac{1}{p_*(0)}}
\log\left(\frac{e}{t}\right)
^{\frac{\theta-1}{p_*(0)}};
\end{align*}
in the final inequality we used the fact that the exponent of the last log term
is negative.   Since $u_*(t)$ is bounded for $t_0<t\leq 1$, we
conclude that $u_* \in L^{p_+),\theta}([0,1])$.
\end{proof}

\bibliographystyle{plain}
\bibliography{grandsmallvar}

\end{document}